\documentclass[12pt]{amsart}
\usepackage{amsmath,amssymb,latexsym,amsthm,amscd,mathrsfs,soul}

\usepackage{verbatim}
\usepackage[active]{srcltx}
\usepackage[all]{xy} \xyoption{arc}\usepackage[left=3cm,top=2cm,right=3cm,bottom = 2cm]{geometry}
\usepackage{graphicx}

\usepackage{physics}

\usepackage{hyperref}
\usepackage{xcolor}
\hypersetup{
    colorlinks,
    linkcolor={red!50!black},
    citecolor={blue!50!black},
    urlcolor={blue!80!black}
}

\theoremstyle{plain}
\newtheorem{thm}{Theorem}[section]

\newtheorem{lem}[thm]{Lemma}
\newtheorem{prop}[thm]{Proposition}
\newtheorem{cor}[thm]{Corollary}

\theoremstyle{definition}
\newtheorem{definition}[thm]{Definition}
\newtheorem{example}[thm]{Example}
\newtheorem{notation}[thm]{Notation}
\newtheorem{bigthm}{Theorem}

\theoremstyle{remark}
\newtheorem{remark}[thm]{Remark}

\usepackage{tikz-cd}
\tikzset{>=latex}

\DeclareMathOperator{\End}{End}

\DeclareMathOperator{\GL}{GL}

\DeclareMathOperator{\Spec}{Spec}

\DeclareMathOperator{\Rank}{rank}

\DeclareMathOperator{\U}{U}
\def\:{\colon\!}

\newcommand{\scr}{\mathscr}

\providecommand{\twomat}[4]{\left(\begin{matrix}#1&#2\\#3&#4\end{matrix}\right)}

\hypersetup{
pdftitle={Rational Local Unitary Invariants of Symmetrically Mixed States of Two Qubits},
pdfauthor={Luca Candelori, Vladimir Y. Chernyak, John R. Klein, and Nick Rekuski}
}

\begin{document}

\title[Rational Local Unitary Invariants of Symmetrically Mixed States]{Rational local unitary invariants of symmetrically mixed states of two qubits }

\author[L.~Candelori]{Luca Candelori$^\ast$}

\author[V.~Y.~Chernyak]{Vladimir Y. Chernyak$^{\ast,\dagger}$}

\author[J.~R.~Klein] {John R. Klein$^\ast$}

\author[N.~Rekuski]{Nick Rekuski$^\ast$}

\thanks{\hskip -.38in{}{\tiny{\sc {$^\ast$}Department of Mathematics, Wayne State University, 656 W. Kirby St, Detroit, MI, 48202 USA}}}
\thanks{\hskip -.38in{}{\tiny{\sc {$^\dagger$}Department of Chemistry, Wayne State University, 5101 Cass Ave, Detroit, MI, 48202 USA}}}
\date{\today}

\begin{abstract}
We compute the field of rational local unitary invariants for locally maximally mixed states and symmetrically mixed states of two qubits. In both cases, we prove that the field of rational invariants is purely transcendental. We also construct explicit geometric quotients and prove that they are always rational.  All the results are obtained by working over the field of real numbers, employing methods from classical and geometric invariant theory over arbitrary fields of characteristic zero. 

\end{abstract}

\maketitle

\setcounter{tocdepth}{1}
{
  \hypersetup{linkcolor=olive}
  \tableofcontents
}

\addcontentsline{file}{sec_unit}{entry}

\section{Introduction} 

\subsection{Background} 
Let $X$ be an affine algebraic variety over a field $k$ that is equipped with an action of a reductive $k$-group $G$.
  Then one may consider $k[X]^{G}$, the algebra of polynomial invariants of the coordinate ring $X$. Unfortunately, the problem of finding a reasonable description of
   $k[X]^{G}$ in terms of generators and relations is typically intractable, despite the recent availability of powerful computational tools.
By contrast, if we pass to the field of {\it rational invariants} the situation often becomes much simpler.
  For example, when $X = \mathbb{A}^{n}$ is affine $n$-space over a field $k$ of characteristic zero, then
$k[\Bbb A^n]^{G}$ may be especially complicated, but in many instances  $k(\Bbb A^n)^{G}$ is a purely transcendental extension of $k$~\cite{dolgachev, CT-S}.
  
A case that naturally arises in quantum mechanics, as well as in quantum computing, is when 
  \begin{itemize}
    \item
      $k=\mathbb{R}$,
    \item
      $X$ is the real affine space of trace one Hermitian (i.e., self-adjoint) operators on the tensor product
       $V_1\otimes V_2$, where $V_1, V_2$ are two-dimensional complex inner product spaces, and
    \item
      $G = \U(V_1) \times \U(V_2)$ is the product of unitary groups acting by conjugation on $X$.
  \end{itemize}
 Elements of the field of rational invariants $k(X)^{G}$ in this case provide a measure of the {\it quantum entanglement} of {\it mixed states} of {\it  two qubits}.
 
The main goal of this paper is to give an explicit presentation for the rational field of invariants for the subspaces of {\em locally maximally mixed states} and {\em symmetrically mixed states}. From a higher perspective, our goal is to lay the foundations for a study of rational local unitary invariants by introducing methods from algebraic geometry and invariant theory. The full calculation of $k(X)^G$ will be described in the forthcoming paper \cite{SecondPaper} by employing the tools developed in this article.

To summarize our results and previous work on the subject, we first fix some notation and terminology. Suppose $\{V_{i}\}_{i=1}^{n}$ are complex $2$-dimensional inner product spaces.
  A \emph{pure state} of {\it $n$-qubits} is a unit vector in the tensor product  $V:=\bigotimes_{i=1}^{n}V_{i}$
  (when $n=2$ we often refer to $2$-qubits as a {\it two qubits}).
  Given a pure state $\psi$, we have an associated projection operator 
  \[
  P_\psi \in \operatorname{End}(V)
  \] 
  defined by 
  $P_\psi(v) = \langle \psi,v\rangle \psi$; this is trace one and Hermitian.
  A convex combination of such projections is called a \emph{mixed state}.  The space of mixed states is not an algebraic
  variety. To obtain a variety we relax convexity and instead consider all real linear combinations of such projections having trace one. This results in 
 the {\it Liouville space} 
 \[
 \scr L \subset \operatorname{End}(V) , 
 \]
  i.e., the real affine space of  trace one Hermitian operators on $V$. 
  
  Let $G = \prod_i \U(V_i)$; this is the {\it local unitary group} of $V$. Note that $G$ acts on $\scr L$ by conjugating operators. Then
  $\Bbb R[\scr L]^G$ is the ring of {\it local unitary polynomial invariants} of mixed states of $n$-qubits.  Note that 
  the $G$-equivariant embedding $V\to \scr L$ induces a homomorphism 
   \[
   \Bbb R[\scr L]^G \to  \Bbb R[V]^G 
   \]
  from the ring of invariants of mixed states to the ring of invariants of pure states.  
  One may also consider $\Bbb R(\scr L)^G$, the field of {\it local unitary rational invariants} of mixed states of $n$-qubits.
   
  \subsection{Historical overview}
  Mixed states of two qubits are used to model errors occurring in faulty quantum channels, so this case is of particular interest in the field of quantum information science~\cite{NielsenChuang}.
  Furthermore, calculating local unitary invariants is important to understand {\it entanglement} of mixed states~\cite{Horodecki}.
  Indeed, local unitary operators correspond physically to quantum evolution that can be applied independently on each system $V_i$.
  Therefore two states that are in the same orbit under local unitary operators can be thought of as having the same level of entanglement.
  Classifying these orbits under local unitary operators is thus equivalent to classifying all the possible levels of entanglement in a composite quantum system.

  Because of the importance of these applications, the ring of local unitary invariants has received considerable attention in recent years in both the mathematical and physical literature.
  We give a brief historical overview:
  \begin{itemize}
    \item{
      Brylinski gave partial results in the case of 3 and 4 pure qubits~\cite{Brylinski},
    }
    \item{
      Meyer and Wallach calculate the Hilbert series in the case of pure qubits and then calculate the invariants in the case of three and four pure qubits~\cite{Wallach-Meyer},
    }
    \item{
      In his book, Wallach recalls several earlier results in the pure qubit case (previously found by Wallach and collaborators) then gives a preliminary study of the mixed qubit state~\cite{Wallach-book},
    }
    \item{
      King, Welsh, and Jarvis give a full presentation of complex local unitary invariants of two mixed qubits~\cite{King}, and
    }
    \item{
      Gerdt, Khvedelidze, and Palii calculate a full presentation for the subspace of `X-states' of two mixed qubits---an important subclass of mixed states~\cite{X-states}.
    }
  \end{itemize}
  However, very little is known about the case of three or more mixed states. It is also worth noting that the above authors work primarily over the algebraically closed field of complex numbers
  rather than over the real numbers.

  In principle, there are general algorithms to calculate generators and relations for the algebra of local unitary invariants~\cite{CIT}.
  However, as the complexity of the problem increases (e.g., increasing the number of qubits) it is impractical to apply these algorithms directly.
  To wit, the authors attempted to replicate the calculations of \cite{King} (the case of two qubits) using the \texttt{InvariantRing} package~\cite{invariantRing} in Macaulay2~\cite{M2} but our system ran out of memory after several hours.
   Moreover, even when a full presentation is given, the explicit generators and relations are too complex to provide meaningful information about the classification of entanglement. 
  To this end, in this article we present several new approaches to the study of local unitary invariants of mixed states of qubits: we consider the rational local unitary invariants, we work with Bloch matrices, and we work exclusively over $\mathbb{R}$.
  
  \subsection{Main results} 
Our focus will be on the rationality question for the field of rational local unitary invariants in the two qubit case. To the best of our knowledge, this article represents the first study in this direction. For simplicity and ease of exposition, we first consider the special families of mixed states given by \emph{locally maximally mixed states} (Section \ref{section:LMMstates}) and \emph{symmetrically mixed states} (Section \ref{section:symStates}), both which are interesting in their own right.
  Broadly, locally maximally mixed states are mixed states where each qubit is as far from being a pure state as possible (Definition \ref{definition:locallyMaximallyMixed}).
  Likewise, symmetrically mixed states are mixed states where the qubits are indistinguishable from each other.
  In both cases, we show that the field of real rational local unitary invariants is  purely transcendental, and we provide explicit expressions for a set of generators.

  To state our results, let $U_2$ be the group of complex unitary $2\times 2$ matrices. In the locally maximally mixed state case the situation is fairly simple. We can actually calculate the ring of polynomial invariants and from there we obtain the following complete description of the field of rational invariants (cf.~Theorem \ref{thm:LMM-invariants} and Corollary  \ref{corollary:rationalInvLMM}):
  
  \begin{bigthm}[LMM States] 
    \label{thm:lmmStatesLetter}
    Let $ \scr L _{M}$ be the affine variety of locally maximally mixed states of two qubits.
    Then there is an isomorphism of graded  $\Bbb R$-algebras
    \[
    \mathbb{R}[ \scr L _{M}]^{ \U_2\times  \U_2} \, \, \cong \,\, \mathbb{R}[t_{2},t_{3},t_{4}]\, ,
    \] where $t_{i}$ is of degree $i$.
    The field $\mathbb{R}( \scr L _{M})^{ \U_2\times  \U_2}$ is the fraction field of $\mathbb{R}[ \scr L _{M}]^{ \U_2\times  \U_2}$, therefore purely transcendental of degree $3$,  generated by $t_2, t_3, t_4$. 
  \end{bigthm}

  For symmetrically mixed states, $V_1 = V_2$ and the group $G$ is given by one copy of  $ \U_2$. This case is more complex than the LMM case, and we prove rationality for this field of invariants without calculating the ring of polynomials invariants: 
   
  \begin{bigthm}[Symmetrically Mixed States]
    \label{thm:symStatesLetter}
    Let $ \scr L _{S}$ be the affine variety of symmetrically mixed states of two qubits. Then the field
    $\mathbb{R}( \scr L _{S})^{ \U_2}$ is purely transcendental of degree $6$.
  \end{bigthm}
  
  Further, we make Theorem B effective in Section \ref{subsec:effective} and explain how to calculate a set of six generators explicitly.  

  The proof of each of the above results employs the {\em Bloch matrix} model of mixed states of two qubits \cite{Gamel}, which is built out of 1- and 2-point correlation functions, measuring expected values of spin measurements along the three different axes.
  We review this construction in Section \ref{section:Bloch_matrix}.
  This model is a generalization of the popular {\em Bloch sphere} model of pure qubits, and it lends itself particularly well to the study of local unitary invariants.
  While the Bloch model is used implicitly in previous work (e.g., \cite{King}) we believe that this is the first article which explicitly employs it to obtain results about local unitary invariants.
  In particular, using the Bloch model we construct in Theorem \ref{thm:geometricQuotientSym} explicit affine geometric quotients (in the sense of \cite{GIT}) for the 
  action of the local unitary group on the variety of symmetric mixed states of two qubits.
  Furthermore, we show (cf.~Example \ref{example:restrictionCalc}) how these geometric quotients can be used, in conjunction with computer-aided calculations and the graphical calculus of \cite{King}, to provide explicit generators for the (purely transcendental) field of rational invariants.
  
 Compared to previous work, in this article we work exclusively over the ground field $k= \mathbb{R}$, and we avoid base-change to the field of complex numbers. As shown in \cite{GIT} and \cite{CT-S}, many results of invariant theory over the the field of complex numbers go through for an arbitrary field of characteristic zero, at the price of replacing standard arguments from classical algebraic geometry (\cite{PV}) with a more scheme-theoretic approach (\cite{CT-S}). Working over the real numbers and expressing the rational invariants in terms of the Bloch matrix makes our work immediately ready to be implemented at the experimental level. 
 
For the sake of brevity we omitted in this article the computation of the field of rational invariants for the full space of two mixed qubits. This field can be calculated  by similar methods (it is also purely transcendental) but the modifications required are lengthy enough that merit their own separate treatment, which we provide in \cite{SecondPaper}.  
 
We believe that this new approach to local unitary invariants is relevant for practical applications. Physicists are not necessarily committed to polynomial invariants: important local unitary invariants such as entropy and concurrence are not polynomials. Therefore, it is likely that a complete classification of the rational local unitary invariants could lead to meaningful applications in quantum physics.

\subsection*{Acknowledgements}
  The authors are supported by the U.S. Department of Energy, Office of Science, Basic Energy Sciences, under Award Number DE-SC-SC0022134.
  The last author is also partially supported by an OVPR Postdoctoral Award at Wayne State University. 

\section{Quotients, sections and rationality}
\label{sec:sections}
  We begin by collecting some necessary results from invariant theory.
  In particular, we need results concerning the rationality of fields of rational invariants, in addition to computational methods to calculate rational invariants explicitly (i.e., the method of \emph{relative sections}).
  When the ground field is of characteristic zero (but not necessarily algebraically closed) these results appear in \cite{CT-S}.
  Since our applications are when $k = \mathbb{R}$, this assumption is essential.
  For the algebraically closed case, the older text of Popov and Vinberg \cite{PV} can also serve as reference. 

\subsection{Algebraic, rational and geometric quotients}
  Suppose $k$ is a field of characteristic zero.
  Let $X/k$ be a geometrically integral affine variety and $G/k$ a reductive algebraic group acting (algebraically) on $X$.
  We write the coordinate ring of $X$ by $k[X]$.
  We denote the ring of polynomial $G$-invariant function on $X$ by $k[X]^G$, so $k[X]^{G}$ is a $k$-subalgebra of $k[X]$.
  Similarly, we write the field of rational $G$-invariant functions by $k(X)^{G}$.
  The rings $k[X]^{G}$ and $k(X)^{G}$ arise as functions on different variations of quotients.

  \begin{definition}
    \label{definition:algebraicQuotient}
    We define the \emph{algebraic quotient} of $X$ via $G$, written $X/\!/ G$, to be the affine scheme $\Spec(k[X]^G)$. 
  \end{definition}

  By construction, polynomial invariants $k[X]^G$ are polynomial functions on the algebraic quotient $X/\!/ G$.
  Furthermore, since $G$ is reductive, by \cite[Thm. 2.2.10]{CIT}, $k[X]^G$ is finitely generated and so $X/\!/ G$ is of finite type.
  In other words, $X/\!/ G$ is an integral $k$-variety.
  By \cite[Def. 0.7]{GIT}, the affine quotient is a {\em categorical quotient} and so $X/\!/ G$ is unique up to unique isomorphism.

  One may hope $k(X)^{G}$ is the fraction field of $X/\!/ G$, but this is not the case in general (c.f. Remark \ref{remark:fractionFields}).
  For example, if $X=\mathbb{A}^{2}_{k}$ and $\mathbb{G}_{m}$ acts diagonally then $(k[X]^{G})_{(0)}=k$ while $k(X)^{G}=k[t]$.
  Therefore, $k(X)^{G}$ cannot always be interpreted as rational functions on $X/\!/ G$.
  
  Since $k(X)^{G}$ is a subfield of $k(X)$, $k(X)^{G}$ is finitely generated.
  Thus, let $y_1, y_{2}, \ldots, y_n$ be generators for $k(X)^G$ as a $k$-algebra.
  Then $Y=\Spec( k[y_1, \ldots, y_n])
$ is a {\em model} for $k(X)^G$, that is, $Y$ is a geometrically integral affine algebraic $k$-variety such that $k(Y)=k(X)^G$.
  We use this property to give our second variation of a quotient.

  \begin{definition}
    A $k$-variety $Y$ is a \emph{rational quotient} of $X$ by $G$ if there is an isomorphism $k(Y)\cong k(X)^G$. 
  \end{definition}
  
  By construction, rational $G$-invariant functions are rational functions on a rational quotient $Y$.
  However, in contrast to the case of the algebraic quotient, there is no canonical choice of rational quotient $Y$, since it is only defined up to birational equivalence. 
Observe that the the injection $k(Y)\cong k(X)^{G}\subseteq k(X)$ induces a dominant rational map $f:X\dashrightarrow Y$.

  The geometric points of an algebraic (or rational) quotient need not coincide with the orbits of the $G$-action.
  However, this is the case for a \emph{geometric quotient}:
  
  \begin{definition}\cite[Def. 2.7]{CT-S}
    A \emph{geometric quotient} of $X$ by $G$ is a $k$-variety $Y$ together with a $k$-morphism $f:X\rightarrow Y$ satisfying 
    \begin{itemize}
      \item[(i)]{
        $f$ is open, constant on $G$-orbits, and induces a bijection $X(\bar{k})/G(\bar{k})\leftrightarrow Y(\bar{k})$
      }
      \item[(ii)]{
        for any open set $V\subseteq Y$, the natural morphism $k[V] \rightarrow k[f^{-1}(V)]^G$ is an isomorphism.
      }
    \end{itemize}
    where $\overline{k}$ is an algebraically closed field containing $k$.
    
    If a geometric quotient exists, we write it as $Y=X/G$.
  \end{definition}
  
  A geometric quotient is always a categorical quotient \cite[Prop. 0.1]{GIT} so $X/G$ is unique up to unique isomorphism.

  Even in our case when $X$ is affine and $G$ reductive, a geometric quotient need not exist (e.g., $\mathbb{A}^{2}_{k}$ with the diagonal action of $\mathbb{G}_{m}$).
  However, by Rosenlicht's theorem, we may choose a $G$-invariant dense open subset $U\subseteq X$ admitting a geometric quotient $U\rightarrow U/G$~\cite{Rosenlicht}.
  With this choice, the $k$-variety $Y = U/G$ is also a rational quotient of $X$ by $G$, by property (ii) of geometric quotients.
  In other words, we can always choose a rational quotient $Y$ that is also a geometric quotient of a dense open $G$-invariant subset of $X$.
  This idea will be important in our applications, so we make the following definition:

  \begin{definition}
    \label{def:geometricRationalQuotient}
    A \emph{geometric rational quotient} of $X$ by $G$ is a geometric quotient of the form $Y = U/G$, where $U\subseteq X$ is an open dense $G$-invariant subset. 
  \end{definition}

\begin{remark} 
\label{rmk:closedAction}
When $X$ is affine, a geometric quotient $X/G$ exists if and only if the  $G$-action is {\em closed}, that is, the geometric $G$-orbits on $X(\bar{k})$ are closed \cite[Section I.2, Amplification 1.3]{GIT}. In this situation, 
$$
 X/G = X/\!/ G  = \Spec(k[X]^G)
$$
and $X/G$ is also a rational quotient. So all three notions of quotients coincide in this particular case. 
\end{remark}

\subsection{Rationality }
  A fundamental question regarding the field of rational invariants $k(X)^G$ is whether it is \emph{purely transcendental}.
  In other words, whether $k(X)^{G}$ is generated by a finite number of algebraically independent rational functions.
  Geometrically, this is equivalent to asking whether a rational quotient $X\dashrightarrow Y$ is \emph{rational}, that is, birationally equivalent to a projective space $\mathbb{P}_k^n$, for some $n$.
  The monograph \cite{CT-S} contains a detailed account of what is known about this question. For our applications we shall only require a variant of the `No-Name Lemma' (according to Popov and Vinberg \cite{PV} it is difficult to establish the authorship of this result). This lemma gives a criterion for the rationality of a reducible representation in the case when one of its summands has {\em almost free} action, as defined below:
  
    \begin{definition}
    \label{def:almostFree}
    The action of $G$ on $X$ is \emph{almost free} if there exists an open dense $G$-invariant subset $U\subseteq X$ such that for every geometric point $x \in U(\bar{k})$, the geometric stabilizer $G_x \subseteq G(\bar{k})$ is trivial. 
  \end{definition}

  The following versions of the No-Name Lemma appear in \cite{dolgachev} and \cite[2.13]{PV} for the case $k=\bar{k}$ and in \cite{CT-S} for the general case of $k$ a field of characteristic zero. 

  \begin{thm}[{\cite[Cor. 3.8]{CT-S}}]
    \label{thm:noName}
    Let $k$ be a field of characteristic zero and $G$ a reductive $k$-group acting almost freely on an affine $k$-variety $X$.
    Further assume $\overline{X}=X\times_k \Spec(\bar{k})$ is factorial (i.e., $\bar{k}[\overline{X}]$ is a unique factorization domain).
    Let $V$ be a finite-dimensional $k$-vector space together with a linear $G$-action.
    Then the field $k(X\times V)^G$ is purely transcendental over $k(X)^G$, of transcendence degree $\dim V$.
    In particular, if $k(X)^G$ is purely transcendental then $k(X\times V)^G$ is also purely transcendental.

  \end{thm}

  \begin{cor}
  \label{cor:rationality}
    Let $k$ be a field of characteristic zero and $G$ a reductive $k$-group acting linearly on a $k$-vector space $W$.
    Suppose that $W = V_1\oplus V_2$ can be decomposed into a direct sum of $G$-invariant subspaces $V_1, V_2$, such that the action of $G$ restricted to $V_1$ is almost free and the field $k(V_1)^G$ is purely transcendental. Then $k(W)^G$ is purely transcendental and of transcendence degree $\dim V_2$ over $k(V_1)^G$. 
  \end{cor}

\begin{proof}
This follows from Thm. \ref{thm:noName} by setting $X = V_1$ and $V = V_2$. Indeed $\overline{V}_{1}$ is factorial, since the coordinate ring of $\overline{V}_{1}$ is a polynomial ring in finitely many variables over the field $\bar{k}$, and therefore a unique factorization domain.
\end{proof}

\subsection{Sections}
  The calculation of the field of rational invariants can be simplified by finding a \emph{relative section} $S\subseteq X$ (also called a \emph{slice}).
  When $k$ is algebraically closed and of characteristic zero this notion can be found in \cite[Section 2]{PV}.
  Here we present a generalization to the case where $k$ is an arbitrary field of characteristic zero which is discussed in \cite[Section 3.1]{CT-S}.
  As in the previous section, let $k$ be a field of characteristic zero, $X$ a geometrically integral affine $k$-variety, and $G$ a reductive $k$-group acting on $X$.

  \begin{definition}
    \label{def:relativeSection}
    A geometrically integral $k$-subvariety $S\subseteq X$ is a \emph{ relative section} (or \emph{slice}) for the $G$-action if there exists a $G$-invariant dense open subset $X_0 \subseteq X$ satisfying the following properties:
    \begin{itemize}
      \item[(i)]{
        The Zariski closure of the orbit space $GS$ is all of $X$ (that is, $\overline{GS} = X$),
      }
      \item[(ii)]{
        Let $N=N(S)=\{g\in G:gS \subseteq S\}$ be the normalizer of $S$, and let $S_0=X_0\cap S$.
        Then $gS_0(\bar{k})\cap S_0(\bar{k})\neq\emptyset$ implies that $g\in N(\bar{k})$.
      }
    \end{itemize}
  \end{definition}

  If $S\subseteq X$ is a relative section, the restriction homomorphism $k[X]^G\to k[S]^N$ is well-defined and injective: if $f\in k[X]^{G}$ restricts to zero on $S$, then it must be zero on all of $X$ by property (i) of Definition \ref{def:relativeSection}.
  It follows that restriction gives an injective field homomorphism $k(X)^G \to k(S)^N$.
  The following consequence of relative sections is that $k(X)^{G}\to k(S)^{N}$ is an isomorphism.

  \begin{thm}[{\cite[Thm 3.1]{CT-S}}]
    \label{thm:relativeSectionsAreBirational}
    Let $S\subseteq X$ be a relative section for the $G$-action.
    The restriction map $k(X)^G \rightarrow k(S)^{N(S)}$ is an isomorphism. 
  \end{thm}
  
  Furthermore, if the relative section $S\subseteq X$ satisfies $\dim(S)=\dim(X)-\dim(G)$, then the normalizer $N(S)$ is finite.
  In this case, the calculation $k(S)^{N(S)}$ is easier.

  Under special circumstances, it is possible to calculate the ring of polynomial invariants $k[X]^G$ from a relative section---not just the rational invariants!
  In this case we say the section is of `Chevalley' type, since it generalizes the Chevalley Restriction Theorem in the theory of semi-simple complex Lie groups (see \cite[3.8]{PV}).  

  \begin{definition}
    \label{def:ChevalleySec}
    A relative section $S\subseteq X$ is called a \emph{Chevalley section} if the restriction homomorphism $k[X]^G\to k[S]^{N(S)}$ is an isomorphism. 
  \end{definition}

  In practice, if $N(S)$ does not act faithfully on $S$ then it is still difficult to calculate the polynomial and rational invariants.
  To avoid this scenario, we can always pass to the Weyl group which always acts faithfully.
  
  \begin{definition}
    \label{remark:WeylGroup}
    Suppose $S\subseteq X$ is a relative section for the $G$-action.
    
    We define the \emph{centralizer} of $S$ by
    \[
      Z(S)=\{g\in G: gs=s \ \forall s\in S\}.
    \]
    
    We also define the \emph{Weyl group} of $S$ by $W(S)=N(S)/Z(S)$.
  \end{definition}
  
  By construction $W(S)$ acts faithfully on $S$ and there are isomorphisms $k(S)^{N(S)}\cong k(S)^{W(S)}$ and $k[S]^{N(S)}\cong k[S]^{W(S)}$.

\subsection{Invariants and localization}

As noted earlier, the field of rational invariants $k(X)^G$ need not coincide with the fraction field of the ring $k[X]^G$. It is possible however to impose conditions on the group $G$ to ensure that this is the case (\cite[Lemma 2.2]{CT-S}). We end this section by generalizing this result to an arbitrary localization by a $G$-invariant set.

  \begin{lem}
  \label{lemma:localizationAndInvariants}
    Let $A$ be an integral $k$-algebra of finite type equipped with an action of a linear algebraic $k$-group $G$.
    Let $S\subseteq A$ be a multiplicatively closed subset not containing $0$ and set $S^{G}=S\cap A^{G}$.
    If $A$ is a UFD, and $G$ is connected with no non-trivial characters, then $(S^{G})^{-1}(A^{G})=(S^{-1}A)^{G}$.

  \end{lem}
  \begin{proof}
    When $k$ is algebraically closed, the same argument as \cite[Lemma 2.2]{CT-S} gives the desired result.
    
    Thus, it suffices to show the non-algebraically closed case reduces to the algebraically closed case.
    With this in mind, consider the short exact sequence
    \[
      0\to (S^{G})^{-1}A^{G}\to (S^{-1}A)^{G}\to C\to 0
    \]
    where $C$ is just the cokernel of the natural inclusion.
    Base changing to the algebraic closure gives the following short exact sequence of $\overline{k}$-modules
    \[
      0\to ((S^{G})^{-1}A^{G})\otimes_{k}\overline{k}\to (S^{-1}A)^{G}\otimes_{k}\overline{k}\to C\otimes_{k}\overline{k}\to 0.
    \]
    On the other hand, since both localization and taking $G$-invariants  commute with base-change to $\bar{k}$, we must have $C\otimes \bar{k} = 0$, by \cite[Lemma 2.2]{CT-S}. This implies that $C = 0$. 
  \end{proof}
  
  As noted in \cite[Remarks after Lemma 2.2]{CT-S}, the assumption that $G$ have trivial character group is necessary.

    \begin{remark}
    \label{remark:fractionFields}
    In the case of a finite group $G$ the proof of Lemma \ref{lemma:localizationAndInvariants} is elementary and the assumption on $A$ being a UFD can be dropped. 
  \end{remark}
  
\section{Preliminaries on Mixed States}
  In this section we recall definitions related to the invariant theory of mixed states of $n$-qubits. First, as noted in the introduction, the space of mixed states is \emph{not} an algebraic variety.
  For this reason, we instead consider the \emph{Liouville space} which is the smallest algebraic $\mathbb{R}$-variety containing all mixed states.
  We also describe the action of unitary matrices and note there is a corresponding action on the Liouville space.
  Second, we describe the \emph{Bloch matrix} representation of the Liouville space.
  Last, we describe locally maximally mixed states and symmetrically mixed states---which each correspond to subvarieties of the Liouville space.

\subsection{Mixed States and Local Unitary Invariants}
  We first briefly recall the mathematical formalism of pure and mixed quantum states as they appear in the context of quantum information science (e.g., \cite{NielsenChuang}).
  As is customary in the subject, we adopt bra-ket notation throughout.
  In this notation, $(V, \braket{\cdot}{\cdot})$ denotes a complex (Hermitian) inner product space, $\ket{\psi}\in V$ a vector, and $\bra{\psi}$ will denote its conjugate dual.

  We say $\ket{\psi}$ is a \emph{pure state} if $\ket{\psi}$ is a unit vector in $(V,\braket{\cdot}{\cdot})$.
  When $V = \bigotimes_{k=1}^n \mathbb{C}^2$ we say a pure state $\ket{\psi}\in V$ is an $n$-qubit, or $n$ qubits.
  More generally, an \emph{ensemble} of pure states is a finite set of pairs $\{(\ket{\psi_i}, p_i)\}$ consisting of pure states $\ket{\psi_i}$ and positive real numbers $0\leq p_i \leq 1$ satisfying $\sum p_i = 1$.
  The real number $p_i$ is physically interpreted as the probability that the ensemble is in state $\ket{\psi_i}$.
  An ensemble can also be represented by the \emph{mixed state} 
  \[
    \rho=\sum p_i \ket{\psi_i}\bra{\psi_i}\in\End(V),
  \]
  which is a positive semi-definite Hermitian linear operator of trace one.
  In coordinates, this is  known as the \emph{density matrix} of the mixed state.
  In this formalism, pure states $\ket{\psi}$ are mixed states corresponding to ensembles supported on only one projector.
  Furthermore, the density matrix of a pure state in this formalism is the rank one projection matrix of the operator $\ket{\psi}\bra{\psi}$. 

  We would like to study the space of mixed states directly, but mixed states do not form an algebraic variety (see the discussion below).
  For this reason, we instead consider the Liouville space.

  \begin{definition}
    Let $\operatorname{Herm}(V)$ be the $\mathbb{R}$-vector space of all Hermitian (i.e., self-adjoint) operators on an inner-product space $V$.
    The \emph{Liouville space} of $V$ is the hyperplane 
    \[
      \mathscr{L}=\mathscr{L}(V)=\{\rho\in\operatorname{Herm}(V):\operatorname{Tr}(\rho)=1\}.
    \]
    
    When no confusion will arise, we refer to elements $\rho\in\mathscr{L}$ as mixed states.
  \end{definition} 
  
  We can view the space of mixed states as a semialgebraic subvariety of the real algebraic variety $ \scr L $~\cite{Gamel}.
  The set of mixed states itself is not an algebraic subvariety of $ \scr L $ because of the positive semi-definite condition on mixed states.
  Nevertheless, the polynomial invariants on $ \scr L $ and its $G$-invariant subvarieties are still of interest in quantum information science.
  Plus, the positivity condition can always be re-introduced later by enforcing bounds and inequalities on the invariants and their values (see Remark \ref{remark:inequalities} below).
  For these reasons, we work with $ \scr L $ instead of the space of mixed states.

  An important property of mixed states is that they are closed under partial trace, that is, the partial trace of a mixed state is another mixed state (in lower dimension).
  This is in contrast to the case of pure states, which are not preserved by taking partial traces.
  
  \begin{definition}
    \label{partialTrace}
    Suppose $V$ is a bi-partite system (i.e., $V=V_{1}\otimes V_{2}$).
    The \emph{partial trace over $V_2$} of a simple tensor is defined as 
  \[
    \operatorname{tr}_{2}\left(\ket{v_1}\bra{v_1'}\otimes \ket{v_2}\bra{v_2'}\right) = \ket{v_1}\bra{v_1'} \cdot  \bra{v_2}\ket{v_2'} \in \End(V_1),
  \]
  which is a mixed state of $V_1$.
  By linearity this definition can be extended to all pure and mixed states.
  \end{definition}
  
   The partial trace over $V_1$ is defined similarly, and it is a mixed state of $V_2$. 
  Let  $\U(V)$ be the unitary group
  of $(V,\braket{\cdot}{\cdot})$   that is, the group of isomorphisms of $V$ that preserver the Hermitian inner product $\braket{\cdot}{\cdot}$.
  Conjugation by $U\in \U(V)$ preserves the Liouville space $ \scr L $ and the set of mixed states inside $ \scr L $.
  Physically, the conjugation action $\rho\mapsto U\rho U^\ast$ represents quantum evolution of the mixed state $\rho$ by $U$ (here $U^*$ denotes the conjugate transpose $\overline{U}^T$).

  \begin{definition}
    For a bi-partite system $V = V_1\otimes V_2$, the subgroup $ \U(V_1)\times \U(V_2)\subseteq \U(V)$ is called the group of {\em local unitary operators}. 
  \end{definition}
  
  In this case, $\U(V_1)\times \U(V_2)$ acts on $V$ by $(U_1, U_2)v=(U_1\otimes  \U_2) v$ and similarly it acts on mixed states $\rho$ by conjugation, i.e.,
  $(U_1,U_2)\cdot \rho := (U_1\otimes  U_2) \rho (U^\ast_1\otimes  U^\ast_2)$.

  In this paper we investigate polynomial and rational invariants of the action via local unitary invariants.
  In quantum physics, these invariants can be thought of as ``measures of entanglement.''\footnote{For a review of the vast literature on the subject, see \cite{Horodecki}.}
 Denote by $\mathbb{R}[ \scr L ]$ the ring of real-valued polynomial functions on $ \scr L $, and by $\mathbb{R}( \scr L )$ its field of fractions, consisting of rational real-valued functions. 

  \begin{definition}
    Suppose $V=V_1\otimes V_2$, $G = \U(V_1)\times \U(V_2)$ and let $ \scr L = \scr L (V)$ be the associated Liouville space.
    We say $f\in\mathbb{R}[ \scr L ]$ is a \emph{local unitary invariant} if $f(gx)=f$ for all $g\in G$ and all $x\in \scr L $.
    We denote the $\mathbb{R}$-algebra of local unitary invariants by $\mathbb{R}[ \scr L ]^G$.
    
    Similarly, we say a rational function $f\in\mathbb{R}( \scr L )$ is a {\it rational local unitary invariant} if $f(gx)=f(x)$ for all $g\in G$ and $x\in \scr L $.
    We write the field of rational local unitary invariants by $\mathbb{R}( \scr L )^G$.
  \end{definition}

  More generally, if $X\subseteq  \scr L $ is a $G$-invariant irreducible subvariety, we can define the ring $\mathbb{R}[X]^G$ of local unitary invariants on $X$ 
  as well as the field $\mathbb{R}(X)^G$ of rational local unitary invariants.
  For the rest of the  article, we study local unitary invariants of mixed states in the case when $ V = V_1\otimes V_2$ is a two-level system consisting of two qubits (i.e., $\dim V_1 = \dim V_2 = 2$), for several different subvarieties $X\subseteq  \scr L  =  \scr L (V)$ that are of interest in quantum information science, including the case $X =  \scr L $.

\subsection{The Bloch matrix representation}
\label{section:Bloch_matrix}
  In order to apply the methods of Section \ref{sec:sections}, we need to recall the \emph{Bloch matrix} representation of a mixed state of two qubits (see e.g., \cite{Gamel}).
Suppose first that $\dim_{\Bbb C} V = 2$ and $V$ is equipped with an orthonormal basis; then we have an identification $V\cong \Bbb C^2$.
  Let 
  \begin{equation}
    \label{eqn:PauliMatrices}
    \sigma_0 = \twomat 1001 , \quad   \sigma_1 = \twomat 0110, \quad \sigma_2 = \twomat 0{-i}i0, \quad \sigma_3 = \twomat {1}00{-1}
  \end{equation}
  be the four Pauli matrices.
  If $\rho\in\End(V)$ is the density matrix representing a mixed state of one qubit, the quantity $\operatorname{tr}(\rho \sigma_i)$ represents the expected value of measuring spin along the axis corresponding to $\sigma_i$.
  As is well-known, the map
  \begin{equation}
    \label{eq:bloch_sphere}
    B: \rho \longmapsto \left( \tr(\rho \sigma_1), \tr(\rho \sigma_2), \tr(\rho \sigma_3) \right)
  \end{equation}
  is a bijection between the space of mixed states of one qubit and the unit ball in $\mathbb{R}^3$ centered at the origin (which is called the \emph{Bloch ball}).
  In this representation, the boundary of the ball consists of the pure states, so the boundary is often called the {\em Bloch sphere} representation of a pure state of one qubit.
  On the other hand, the origin of the ball corresponds to the \emph{maximally mixed} state $\rho = \frac{1}{2}I$.
  The key property of the Bloch ball representation is that unitary evolution on mixed states is equivariant under $B$ with respect to 3-dimensional rotations of the ball.
  In other words, for every $U\in \U(V)$ and every mixed state $\rho \in \End(V)$, $B(U(\rho))=r_{U}(B(\rho))$ where $r_{U}\in\operatorname{SO}_{3}(\mathbb{R}$) is the image of the standard homomorphism
   $r\: \U_2(\Bbb C) \to\operatorname{SO}_{3}(\mathbb{R})$.
  In this way, unitary evolution for a mixed state of one qubit can be `visualized' as a rotation of the corresponding point of the Bloch ball. 

  For mixed states of two qubits there is a convenient generalization of the Bloch ball that is well-suited for the study of local unitary invariants.
  To describe it, we first generalize the map $B$ in terms of \emph{correlation functions}. 

  \begin{definition}
    Let $V = V_1\otimes V_2$ be a bi-partite system of two qubits (i.e., $\dim_{\Bbb C} V_1 = \dim_{\Bbb C} V_2 = 2$). We equip each
    $V_i$ with an orthonormal basis.
    For any pair of indices $(i,j)$ with $i,j \in \{0,1,2,3\}$, the \emph{correlation function} $C_{ij}:  \scr L  \rightarrow \mathbb{R}$ on Liouville space is given by the formula
    \[
      C_{ij}(\rho)=\operatorname{tr}(\rho\cdot \sigma_{i}\otimes\sigma_{j}).
    \]
    When $\rho$ is clear from context, we will drop it from the notation and just write $C_{ij}$.
    
    When either $i=0$ or $j=0$ but not both we say $C_{ij}$ is a \emph{1-point} correlation function, and if $i \ne 0 \ne j$,
     we say that $C_{ij}$ is a \emph{2-point} correlation function.
  A 2-point correlation function $C_{ij}$ satisfying $i\neq j$ is called a \emph{mixed} correlation function. 
  \end{definition}

  From the axioms of quantum mechanics, the value of each correlation function represents the expected value of measuring spin with respect to different choices of axis for each qubit.
  As in \eqref{eq:bloch_sphere}, for each mixed state $\rho\in \scr L $ we can then collect all the correlation functions in  a $4\times 4$ matrix $B(\rho)=(C_{ij})$
   which we call the \emph{Bloch matrix representation} of the $\rho$. 
  
  It will be convenient to repackage this expression.
  First, note $C_{00} = \operatorname{tr}(\rho) = 1$.
  Second, we set $u=(C_{i0})_{1\leq i\leq 3}$, $v =(C_{0j})_{1\leq j\leq 3}$ equal to $1$-point correlation functions and $C=(C_{ij})_{1\leq i,j\leq 3}$ 
  the $3\times 3$ matrix of $2$-point correlation functions.
  
  With this notation, $u$ (resp. $v$) is a mixed state of one qubit corresponding to taking the partial trace of $\rho$ with respect to $V_2$ (resp. $V_1$).
  In other words, $u$ and $v$ lie on two separate Bloch balls.
  Describing the space containing $C$ is more complicated, but it can be shown that it is nevertheless a semi-algebraic variety~\cite{Gamel}.
  Using $u$, $v$, and $C$ we can rewrite the Bloch matrix of $\rho$ in the convenient form
  \[
    B(\rho) = \left( \begin{tabular}{c|c}
 $1$ & $v^t$ \\
 \hline
 $u$ & $C$ \\ 
 \end{tabular} \right)\in\mathbb{R}^{15}.
  \] 

  There is an natural action of $\operatorname{SO}_{3}(\mathbb{R})\times\operatorname{SO}_{3}(\mathbb{R})$ on the Bloch matrix $B(\rho)$.
  Namely, if $(R_1,R_2)$ is a pair of rotations in $\operatorname{SO}_{3}(\mathbb{R})\times \operatorname{SO}_{3}(\mathbb{R})$, we define the action
  \begin{equation}
    \label{eqn:generalAction}
    (R_1, R_2) B(\rho) = \left( \begin{tabular}{c|c}
 1 & $ (R_2v)^t$ \\
 \hline
 $R_1u$ & $R_1CR_2^t$ \\ 
 \end{tabular} \right) .
  \end{equation}
  This action is equivariant with respect to unitary evolution of $\rho$ by local unitary operators.
  In other words, if $(U_1, U_2) \in \U(V_1)\times \U(V_2)$, then
  \[
    B\left((U_1\otimes  \U_2) \rho (U_1\otimes  \U_2)^{-1}\right) = (r(U_1),r( \U_2))B(\rho) ,
  \]
  where again here $r$ is the standard homomorphism $r:\U_2(\Bbb C)\rightarrow \operatorname{SO}_{3}(\mathbb{R})$, as in the case of one mixed state.

  \begin{notation}
    Suppose $V=V_{1}\otimes V_{2}$ and $X\subseteq  \scr L (V)$.
    We write the image of $X$ under the morphism $B: \scr L \to\mathbb{R}^{15}$ by $BX$.
\end{notation}
 
  In particular, $B \scr L  = \mathbb{R}^{15}$, since the 1-point and 2-point correlation functions form an orthonormal basis for the dual of the space of $2\times 2$ and $4\times 4$ Hermitian matrices, respectively.

\subsection{Subvarieties of Mixed States}
  In this subsection, we define the subvarieties $ \scr L _{M}\subseteq \scr L $ of \emph{locally maximally mixed states} and $ \scr L _{S}\subseteq \scr L $ of \emph{symmetrically mixed states}.
  
  We first define locally maximally mixed states.
  Recall from Section \ref{section:Bloch_matrix} that the maximally mixed state of one qubit is the qubit with density matrix $\rho=\frac{1}{2}I$.
  Also recall there are two different maps from a mixed state of two qubits $\rho\in\End(V_1\otimes V_2)$ to a mixed state of one qubit given by the two partial traces 
  $\operatorname{tr}_1$ and $\operatorname{tr}_2$.

  \begin{definition}
    \label{definition:locallyMaximallyMixed}
    The set of \emph{locally maximally mixed states} (abbreviated \emph{LMM states}) is the subset of mixed states $\rho\in\End(V_1\otimes V_2)$ such that both $\operatorname{tr}_1(\rho)$ and $\operatorname{tr}_2(\rho)$ are maximally mixed (i.e., equal to $\frac{1}{2}I$) where $\operatorname{tr}_{i}$ is the partial trace over $V_{i}$ (see Definition \ref{partialTrace}).
  \end{definition}
  
  \begin{remark}
    A well-known class of locally maximally mixed states is given by the \emph{Bell states}.
    These are pure states of two qubits given by expressions such as
    \[
      \ket{\psi}=\frac{\ket{0}\otimes\ket{0} + \ket{1}\otimes\ket{1}}{\sqrt{2}},
    \] 
    where, as is standard in quantum information science, $\{\ket{0}, \ket{1}\} \subseteq V = \mathbb{C}^2$ indicates the orthonormal basis consisting of eigenvectors of the Pauli matrix $\sigma_3$.
    Indeed, it can be easily checked that both partial traces of $\ket{\psi}$ are equal to $\tfrac{1}{2}I$.
    It can also be checked that the Bell states are the only examples of pure states that are maximally mixed~\cite{Gamel}. 
  \end{remark}
    
  As with mixed states, locally maximally mixed states do not form a subvariety of the Liouville space.
  Instead, we consider ``maximally mixed" elements of $ \scr L $ which do not need to satisfy the positive semi-definite property of a mixed state.
    
  \begin{definition}
    \label{definition:locallyMaximallyMixedLiouville}
    We define $ \scr L _{M}\subseteq \scr L $ by
    \[
       \scr L _{M}=\left\{\rho\in \scr L |\operatorname{tr}_{1}(\rho) =\tfrac{1}{2}I = \operatorname{tr}_{2}(\rho) \right\}.
    \]
If no confusion arises, we typically refer to elements of $\rho\in \scr L _{M}$ as locally maximally mixed states.
  \end{definition}
  
  By definition, $ \scr L _{M}$ is an intersection of hyperplanes in Liouville space, so $ \scr L _{M}$ is a real subvariety of $ \scr L $.
  In fact, abstractly $ \scr L _{M}=\mathbb{A}^{9}_{\mathbb{R}}$ so $ \scr L _{M}$ is geometrically irreducible and factorial.

  We now define symmetrically mixed states. Suppose that $V_1 = V_2$. Then
$V = V_1 \otimes V_1$ is equipped with an involution  
 $\tau$  which swaps the tensor factors. The involution of $V$ induces an involution on its Liouville space by operator conjugation.
 
   \begin{definition} Let $V = V_1 = V_2$ be a two-dimensional Hilbert space.
    The set of \emph{symmetrically mixed states}  is the subset of mixed states whose 
    associated density matrix $\rho\in\scr L$ satisfies $\tau\rho=\rho\tau$.
  \end{definition}
  
  As in the case of mixed states and locally maximally mixed states, symmetrically mixed states do not form a subvariety of the Liouville space $ \scr L $.
  To rectify this, we introduce the {\it symmetric Liouville space}:
  \begin{definition}
    \label{definition:symmetricallyMixedStates}
    We define $ \scr L _{S}\subseteq \scr L $ by 
    \[
       \scr L _{S}=\{\rho\in \scr L \mid \tau\rho=\rho\tau\}.
    \]
    i.e., $\scr L_S = \scr L^{\Bbb Z_2}$ is the fixed point set of the canonical involution.
  \end{definition}

 By a straightforward calculation that we omit, a mixed state $\rho$ is symmetrically mixed if and only if
its Bloch matrix representation has the form
  \[
    B(\rho) = \left( \begin{tabular}{c|c}
    $1$ & $v^t$ \\
    \hline
    $v$ & $A$ \\ 
    \end{tabular} \right), 
  \]
 in which $A = A^T$ is a real symmetric $3\times 3$ matrix.
  Therefore, $B \scr L _{S}$ is a 9-dimensional subspace of $B \scr L $, isomorphic to the direct sum of $\mathbb{R}^3$ and the vector space of $3\times 3$ symmetric matrices---which is $6$-dimensional. 

\section{Locally maximally mixed states}
\label{section:LMMstates}
The goal of this section is to determine the full ring of polynomial invariants for the subvariety  $ \scr L_M \subseteq \scr L $ of 
locally maximally mixed states.  Recall 
that $ \scr L _{M}$ is a geometrically integral, factorial real subvariety of $ \scr L$ (cf.~Definition \ref{definition:locallyMaximallyMixedLiouville}).

  Instead of working with $ \scr L _{M}$ directly, we consider the associated Bloch matrix representation.
  Since $\operatorname{tr}_{1}(\rho)$, $\operatorname{tr}_{2}(\rho)=\frac{1}{2}I$ for all $\rho\in \scr L _{M}$, and $\frac{1}{2}I$ is the origin of the Bloch ball,
  we have that $u(\rho),v(\rho)=0$ for all $\rho\in \scr L _{M}$.
  In other words,
  \[
    B(\rho)=\left( \begin{tabular}{c|c}
 $1$ & $ 0$ \\
 \hline
 $0$ & $C$ \\ 
 \end{tabular} \right)\in\mathbb{R}^{9}.
  \]
  In particular, the linear map $B(\rho)\mapsto C$ is an isomorphism of $B \scr L _{M}$ with the 
  real vector space $M_{3}(\mathbb{R})$ of
   $3\times 3$ matrices.
  
  Furthermore, the action of local unitary operators $(U_1, U_2) \in \U(V_1)\times \U(V_2)$ on $B \scr L _{M}$ is given by 
  \begin{equation}
    \label{eqn:LMM-rep}
    (R_1, R_2) B(\rho) = \left(
      \begin{tabular}{c|c}
        $1$ & $0$ \\ \hline
        $0$ & $R_1CR_2^t$ \\ 
      \end{tabular} \right)(\rho) ,
  \end{equation}
  where $R_1 = r(U_1),R_2 = r(U_2) \in \operatorname{SO}_{3}(\mathbb{R})$ are rotation matrices that are the images of the unitary operators $U_1,  U_2$ under the standard map 
  $r:\U(2)\to\operatorname{SO}_{3}(\mathbb{R})$.
  The action of each $R_i$ is clearly linear, and therefore we obtain a linear representation 
\[
\varphi_{M}: \operatorname{SO}_{3}(\mathbb{R})\times \operatorname{SO}_{3}(\mathbb{R}) \longrightarrow \GL(B\mathscr L_M) \cong \GL_9(\mathbb{R}). 
\]
We first find a relative section for this action. Consider the linear subspace of $B\mathscr L_{M}$ defined by
\[
S = \left\{ \left( \begin{tabular}{c|c}
 1 & 0 \\
 \hline
 0 & $C$ \\ 
 \end{tabular} \right) \in B\mathscr L_{M} : C  \text{ is diagonal  }  \right\} .
\]
 In particular, $S$ is an irreducible real algebraic variety.
  
  We next calculate the normalizer $N(S)$ of $S$.
  A pair $(R_1, R_2) \in \operatorname{SO}_{3}(\mathbb{R})\times \operatorname{SO}_{3}(\mathbb{R})$ belongs to $N(S)$ if and only if there exist diagonal matrices $C$, $D$ satisfying
  \[
    R_1 C R_2^t=D.
  \]
  By multiplying with the transpose, the rotation matrices $R_1, R_2$ must satisfy $R_1 CC^t R_1^t = DD^t$ and $R_2^t C^t C R_2 = D^tD$.
  Since $CC^t = C^tC = C^2$ and $DD^t = D^tD = D^2$ are again diagonal matrices, $R_1$ and $R_2$ must belong to the normalizer of the subgroup of diagonal matrices inside $\operatorname{SO}_{3}(\mathbb{R})$.
  This normalizer is equal to the subgroup of signed permutation matrices of determinant +1, a group of order $24$ that is abstractly isomorphic to the \emph{chiral octahedral group} $\operatorname{O}\subseteq\operatorname{SO}_{3}(\mathbb{R})$ the group of orientation-preserving symmetries of the octahedron.
  
  Elements of $\operatorname{O}$ are of the form $EP$, where $E$ is a diagonal matrix with entries in $\{\pm 1\}$ and $P$ is a permutation matrix.
  In this notation, write $R_i = E_iP_i$.
  The pair $(R_1, R_2) \in \operatorname{O}\times\operatorname{O}$ is an element of $N(S)$ (i.e., $R_1CR_2^t$ is diagonal for every diagonal matrix $C$) if and only if $P_1 = P_2$. 
  Therefore,
  \[
    N(S)=\{(E_1P, E_2P)\in\operatorname{O}\times\operatorname{O}\}. 
  \]
  
  Since the action of $N(S)$ on $S$ is not faithful, it is also necessary to compute the Weyl group of $S$ (see Remark \ref{remark:WeylGroup}).
  From the description of $N(S)$ we find
  \[
    Z(S) = \{(E_1P, E_2P) \in \mathrm{O}\times \mathrm{O}: E_1 = E_2 \}. 
  \]
  Therefore, we find $W(S)$ is abstractly isomorphic to $\operatorname{O}$.

  We now prove $S$ is a relative section (Definition \ref{def:relativeSection}). 
  \begin{prop}
    \label{prop:relativeSectionLMM}
  $S\subseteq B\mathscr L_M$ is a relative section for the action of $\operatorname{SO}_{3}(\mathbb{R})\times \operatorname{SO}_{3}(\mathbb{R})$. 
  \end{prop}
  \begin{proof}
    For ease of notation, set $G = \operatorname{SO}_{3}(\mathbb{R})\times \operatorname{SO}_{3}(\mathbb{R})$.
    
    First, we claim the orbit 
    \[
      GS = \{ M \in B \scr L _{M}  : M = R_1CR_2^t \ \text{for some} \ C\in S \ \text{and} \ R_1, R_2 \in \operatorname{SO}_{3}(\mathbb{R})\} 
    \]
is all of $B \scr L _{M}$.
  To see this, by the singular-value decomposition of the matrix $M \in B \scr L _{M}$, there are orthogonal matrices $R_1, R_2 \in\operatorname{O}_{3}(\mathbb{R})$ such that
$M  = R_1\Sigma R_2^t$
where $\Sigma$ is a diagonal matrix with non-negative entries (the singular values of $M$). Choose diagonal matrices $D_1, D_2$ whose entries belong to $\{\pm 1 \}$ and such that $\det D_i = \det R_i \in \{ \pm 1 \}$. Then 
\[M  = (R_1D_1)(D_1^{-1}\Sigma D_2^{-1}) (R_2D_2)^t\]
so that $M$ is in the same $\operatorname{SO}_{3}(\mathbb{R})\times \operatorname{SO}_{3}(\mathbb{R})$-orbit of the matrix $(D_1^{-1}\Sigma D_2^{-1})$. The latter matrix is diagonal, and therefore it belongs to $S$. This shows that $GS = X$. Next, let 
\[
X_0 = \left\{ \left( \begin{tabular}{c|c}
 1 & 0 \\
 \hline
 0 & $C$ \\ 
 \end{tabular} \right)  : CC^t \text{ has distinct eigenvalues} \right\} 
\] 
This is a $G$-invariant dense open subset of $X$, being the complement of the zero locus of the discriminant $\Delta$ of the characteristic polynomial of $CC^t$ (note that this is a polynomial $G$-invariant). Let 
\[
S_0 = X_0\cap S = \left\{ \left( \begin{tabular}{c|c}
 1 & 0 \\
 \hline
 0 & $C$ \\ 
 \end{tabular} \right)  : C \text{ diagonal}, C^2 \text{ has distinct eigenvalues} \right\}.
\]  
Suppose $C_1, C_2$ represent two elements of $S_0(\mathbb{C})$ such that $R_1C_1R_2^t = C_2$, for $R_1, R_2 \in \mathrm{SO}_3(\mathbb{C})$. By multiplying with the transpose, we see that both $R_1C^2_1R_1^t$ and $R_2C^2_1R_2^t$ are equal to $C_2^2$, a diagonal matrix with distinct entries, so that $R_1, R_2$ must be signed permutation matrices belonging to $N(S)(\mathbb{C})$. Therefore $S$ is a relative section. 
\end{proof}

By inspecting the proof of Proposition \ref{prop:relativeSectionLMM}, we are able to construct a geometric rational quotient \ref{def:geometricRationalQuotient} for the action of $\operatorname{SO}_{3}(\mathbb{R})\times \operatorname{SO}_{3}(\mathbb{R})$ on LMM states:

\begin{thm}
\label{thm:geometricQuotientLMM}
Let $X = B\mathscr L_M$, $G = \operatorname{SO}_{3}(\mathbb{R})\times \operatorname{SO}_{3}(\mathbb{R})$ and let 
\[
X_0 = \left\{ \left( \begin{tabular}{c|c}
 1 & 0 \\
 \hline
 0 & $C$ \\ 
 \end{tabular} \right)  : CC^t \text{ has distinct eigenvalues} \right\} \subseteq X. 
\] 
Let $\Delta$ be the discriminant of the characteristic polynomial of $CC^t$. Then the algebraic quotient \[X_0/\!/ G = \Spec( \mathbb{R}[X_0]^G ) = \Spec( \mathbb{R}[X]^G[\Delta^{-1}])\]
is a geometric  quotient of $X_0$, and therefore a geometric rational quotient for the action of $G$ on $X$. 
\end{thm}

\begin{remark}
\label{rmk:localizationAndInvariants}
Note that in identifying the quotient explicitly we have used the fact that 
$$
\Spec( \mathbb{R}[X]^G[\Delta^{-1}]) = \Spec( \mathbb{R}[X,\Delta^{-1}]^G),
$$
which follows from Lemma \ref{lemma:localizationAndInvariants} since $G$ is connected, has no non-trivial characters, and $X$ is factorial. 
\end{remark}

For the proof of this theorem and later in the article, we shall require a variant of the
singular value decomposition for complex matrices:

\begin{lem}
\label{lemma:SVDcomplex}
Let $C$ be an  $n \times n$  matrix over the complex numbers such that $CC^t$ has distinct eigenvalues. Then: 
\begin{itemize}
\item[(i)]There exist orthogonal matrices $P,Q \in \mathrm{O}_n(\mathbb{C})$ such that $P C Q^t$ is a diagonal matrix.
\item[(ii)] If $M$ is another such matrix such that $MM^t$ and $CC^t$ are conjugate, then there exist orthogonal matrices $P,Q \in \mathrm{O}_n(\mathbb{C})$ such that $PCQ^t = M$. 
\end{itemize}
\end{lem}

\begin{proof} As  the proof of (ii) is straightforward, we will only prove (i). Denote the rank of a matrix $A$ by $\Rank(A)$.
By \cite[Thm.~2]{Choudury-Horn} (cf.~\cite[Thm.~3]{Craven}),  it will suffice to show that $\Rank(CC^t) = \Rank(C)$. 
Since $CC^t$ has distinct eigenvalues, we infer that $\Rank(CC^t)$ is either $n$ or $n-1$. 
If $\Rank(CC^t) = n$, then $CC^t$ is invertible, and consequently so is $C$, so  $\Rank(CC^t) = \Rank(C)$ in this case.

When \ $\Rank(CC^t) = n-1$, it follows that $(\det C)^2 = \det (CC^t) = 0$ and therefore $\det C = 0$ as well. Consequently,
\[
\Rank(CC^t) = n-1 \le \Rank(C)
\]
(since $\Rank(AB) \leq \min (\Rank(A), \Rank(B))$. Hence, $\Rank(C)$ is either $n-1$ or $n$. But $\Rank(C) \ne n$ as $\det C = 0$. Hence,
$\Rank(C) = n-1 = \Rank(CC^t)$. \qedhere
\end{proof}

\begin{proof}[Proof of Theorem \ref{thm:geometricQuotientLMM} ]
Because $X_0$ is affine, it suffices to show that the orbits of $G(\mathbb{C}) = \mathrm{SO}_3(\mathbb{C})\times \mathrm{SO}_3(\mathbb{C}) $ on $X_0(\mathbb{C})$ are closed (\cite[I.1.3]{GIT}).  If $C \in X_0(\mathbb{C})$ represents a geometric point, we  prove that its orbit is closed by showing that it corresponds to the zero set of the characteristic polynomial of $CC^t$. Indeed, this orbit consists of elements of the form $M = R_1CR_2^t$, with $R_1, R_2 \in \mathrm{SO}_3(\mathbb{C})$. For any such $M$, the matrix $MM^t$ is conjugate to $CC^t$, and therefore it is annihilated by the characteristic polynomial of $CC^t$. Conversely, suppose that a matrix $M$ has the property that $MM^t$ has distinct eigenvalues and is annihilated by the characteristic polynomial of $CC^t$. Then the minimal polynomial of $MM^t$ is equal to the characteristic polynomial and it divides the characteristic polynomial of $CC^t$, so that the characteristic polynomial of $MM^t$ must be equal to the characteristic polynomial of $CC^t$. Since the matrices have distinct eigenvalues, this means that $CC^t$ is conjugate to $MM^t$. by Lemma \ref{lemma:SVDcomplex}, there exist $P,Q \in \mathrm{O}_3(\mathbb{C})$ such that $PCQ^t = M$, and moreover we may assume that $\det P = \det Q = 1$, as in the proof of Proposition \ref{prop:relativeSectionLMM}. Therefore $C$ and $M$ are in the same $G$-orbit, and we have proved that the $G$-orbit of $C$ corresponds to the zero locus of the characteristic polynomial of $CC^t$, a closed subset of  $X_0(\mathbb{C})$. 
\end{proof}

Next, we calculate the ring of polynomial invariants. To construct the invariants explicitly, recall that the Bloch matrix representation of LMM states is of the form 
\[
B(\rho) = \left( \begin{tabular}{c|c}
 1 & 0 \\
 \hline
 0 & $C$ \\ 
 \end{tabular} \right)(\rho) \in \mathbb{R}^{9}.
\]
where $C$ is the 3-by-3 matrix of 2-point correlation functions. Let 
\begin{equation}
\label{eqn:LMMinvariants}
t_2 = \tr CC^t, \quad t_3 = \det C, \quad t_4 =  \tr (CC^t)^2. 
\end{equation}
It is easy to check that these are polynomial functions in $\mathbb{R}[\mathscr L_{M}]$ that are invariant under the action of $\operatorname{SO}_{3}(\mathbb{R})\times \operatorname{SO}_{3}(\mathbb{R})$ and are of degrees 2,3 and 4 respectively. 

\begin{thm}
\label{thm:LMM-invariants}
The ring of invariants of LMM states is the polynomial ring $\mathbb{R}[t_2,t_3, t_4]$, where each $t_i$ is graded in degree $i$, $i = 2,3,4$.  
\end{thm}

\begin{proof}
We prove the theorem by showing that $S$ is Chevalley section (Definition \ref{def:ChevalleySec}) for the action of $G = \operatorname{SO}_{3}(\mathbb{R})\times \operatorname{SO}_{3}(\mathbb{R})$ on $X = B\mathscr L_M$. Since we proved already that $S$ is a relative section in Proposition \ref{prop:relativeSectionLMM}, it suffices to show that the restriction homomorphism of invariants
\[
\mathrm{res}_S: \mathbb{R}[X]^G \longrightarrow \mathbb{R}[S]^{W(S)} 
\]
is an isomorphism. To show injectivity,  note that if a polynomial invariant  $f\in \mathbb{R}[V]^G$ satisfies $\mathrm{res}_S(f) = 0$, then the value of this invariant must be zero on every element of $GS$, since $f$ is constant on orbits. But $GS=X$ (as shown in the proof of Proposition \ref{prop:relativeSectionLMM}) , therefore $f = 0$ on all of $X$ and $\mathrm{res}_S$ is injective. Next, we must prove that $\mathrm{res}_S$ is also surjective. This can be done by using our knowledge of the Weyl group $W(S)$ and by computing the ring of invariants $\mathbb{R}[S]^{W(S)}$ directly. Consider the invariants 
\[
s_1 = x_1^2 + x_2^2 + x_3^2, \quad s_2 = x_1x_2x_3, \quad s_3 = x_1^4 + x_2^4 + x_3^4 \in \mathbb{R}[S]^{W(S)}
\]
of degrees 2,3 and 4 respectively. By a simple calculation, 
\[
\det (\partial s_i/\partial x_j)_{i,j = 1,2,3} = 8(x_1^2x_3^4 - x_1^2x_2^4 - x_2^2x_3^4 + x_1^4x_2^2 + x_3^2x_2^4 - x_1^4x_3^2) \neq 0
\]
and $\deg s_1\cdot \deg s_2 \cdot \deg s_3 = 24 = |W(S)|$, so that \[
\mathbb{R}[S]^{W(S)} = \mathbb{R}[s_1, s_2, s_3]
\]
is a polynomial ring generated by $s_1, s_2, s_3$ \cite[Theorem 3.9.4]{CIT}. These three invariants correspond to the restriction of the $G$-invariants $t_1 = \tr CC^t$, $t_2 = \det CC^t $ and $t_3 = \tr (CC^t)^2$. It follows that the restriction homomorphism $
\mathrm{res}_S$ is also surjective (and degree-preserving), and therefore it is an isomorphism of graded $\mathbb{R}$-algebras.  
\end{proof}

\begin{remark}
\label{remark:inequalities}
Theorem \ref{thm:LMM-invariants}, which is a result about polynomial invariants, does not take into account the positivity condition on mixed states. The  positivity condition imposes restrictions on the values of the three invariants $t_1, t_2, t_3$. In particular, using the inequalities of \cite{Gamel} we deduce the bounds 
$0\leq t_2 \leq 3,\quad t_3 \leq (1-t_2)/2 ,\quad 0 \leq t_4 \leq -2t_3 +(1-t_2)^2/4$.
\end{remark}

We can also prove that the rational field of invariants is rational of transcendence degree 3, and provide explicit generators:

\begin{cor} \label{corollary:rationalInvLMM}
The field of rational invariants of LMM states is isomorphic to $\mathbb{R}(t_1, t_2, t_3)$. 
\end{cor}

\begin{proof}
Let $G = \operatorname{SO}_{3}(\mathbb{R})\times \operatorname{SO}_{3}(\mathbb{R})$ and $X = B\mathscr L_M$ and let $S\subseteq X$ be the Chevalley section found in the proof of Theorem \ref{thm:LMM-invariants}. Restriction of invariants gives an isomorphism $\mathbb{R}[X]^G \cong \mathbb{R}[S]^{W(S)}$, the latter which is a polynomial ring in three variables.  Because $W(S)$ is finite, the field of rational functions $\mathbb{R}(S)^{W(S)}$ is the fraction field of $\mathbb{R}[S]^{W(S)}$ (Remark \ref{remark:fractionFields}), so it is purely transcendental and of transcendence degree 3.  By Theorem \ref{thm:relativeSectionsAreBirational}, it follows that $\mathbb{R}(X)^G$ must also be purely transcendental on three generators. These generators can be chosen to be $t_1, t_2, t_3$, since these polynomials are algebraically independent. 
\end{proof}

\section{Symmetrically mixed states}
\label{section:symStates}

We now consider the subvariety $X = \mathscr L_S\subseteq \mathscr L$ of symmetrically mixed states;  for convenience, we sometimes refer to these as {\it symmetric states}. Recall from Definition \ref{definition:symmetricallyMixedStates} that
such mixed states are characterized by the property that their density matrix commutes with the twist map 
$\tau: V\otimes V \rightarrow V\otimes V$.

By a straightforward calculation that we omit, a mixed state $\rho$ lies in $\scr L_S$ if and only if 
its Bloch matrix representation is of the form 
\[
B(\rho) = \left( \begin{tabular}{c|c}
 1 & $v^t$ \\
 \hline
 $v$ & $A$ \\ 
 \end{tabular} \right), 
\] 
in which $A \in \text{M}_3(\Bbb R)$ is symmetric, i.e., $A = A^T$. 
It follows that $B\scr L_S$ is a 9-dimensional subspace of $B\mathscr L$, isomorphic to  a direct sum of $\mathbb{R}^3$ and the vector space of $3\times 3$ symmetric matrices, which is 6-dimensional. 

Arbitrary local unitary operators do not preserve symmetric states. Equivalently, the subspace $B\scr L_S$ is not invariant under the $\operatorname{SO}_{3}(\mathbb{R})\times \operatorname{SO}_{3}(\mathbb{R})$ action $\varphi$. Instead, we get an action of $\operatorname{SO}_{3}(\mathbb{R})$ embedded diagonally inside the product $\operatorname{SO}_{3}(\mathbb{R})\times \operatorname{SO}_{3}(\mathbb{R})$. More precisely, the representation on symmetric states 
$
\varphi_S: \operatorname{SO}_{3}(\mathbb{R}) \longrightarrow GL(B\scr L_S)
$
is given by 

\[
R\cdot \left( \begin{tabular}{c|c}
 1 & $v^t$ \\
 \hline
 $v$ & $A$ \\ 
 \end{tabular} \right) = 
 \left( \begin{tabular}{c|c}
 1 & $(Rv)^t$ \\
 \hline
 $Rv$ & $RAR^t$ \\ 
 \end{tabular} \right) .
\]

We now construct a relative section (Definition \ref{def:relativeSection}) for the action of $\operatorname{SO}_{3}(\mathbb{R})$ on symmetric states. Let 
\[
S = \left\{ \left( \begin{tabular}{c|c}
 1 & $v^t$ \\
 \hline
 $v$ & $A$ \\ 
 \end{tabular} \right) \in B\mathscr L_S: A \text{ is diagonal} \right \}.
\]
This is a six-dimensional subspace of $B\scr L_S$  whose Weyl group $W(S)$ is the normalizer of the diagonal matrices inside $\operatorname{SO}_{3}(\mathbb{R})$. As in Section \ref{section:LMMstates}, this group can be identified with the chiral octahedral group $\mathrm{O}$ of $3\times 3$ signed permutation matrices of determinant 1. We first prove that the restriction map on invariants is injective: 

\begin{prop}
\label{prop:injectionSymStates}
The restriction homomorphism $\mathrm{res}_S: \mathbb{R}[B\scr L_S]^{\operatorname{SO}_{3}(\mathbb{R})} \rightarrow \mathbb{R}[S]^{W(S)} $ is injective. 
\end{prop}

\begin{proof}
Let $A$ be a symmetric matrix. By the Spectral Theorem, there is an orthogonal matrix $U$ such that $UAU^T = D$ is diagonal. As in the proof of Proposition \ref{prop:relativeSectionLMM}, we can multiply $U$ by a diagonal matrix $D_0$ consisting of only $\pm 1$ such that $\det (D_0U) = 1$, so that $D_0DD_0^T = D = (D_0U)A(D_0U)^T$. Therefore we may assume that $U = R$ belongs to $\operatorname{SO}_{3}(\mathbb{R})$, so that every $\operatorname{SO}_{3}(\mathbb{R})$-orbit in $B\scr L_S$ intersects $S$ in at least one point. Suppose now that $f \in \mathbb{R}[B\scr L_S]^{\operatorname{SO}_{3}(\mathbb{R})}$ is a polynomial invariant such that $\mathrm{res}_S(f) = 0$. Then $f = 0$, since every $\operatorname{SO}_{3}(\mathbb{R})$-orbit contains a point of $S$, and the value of $f$ is constant on $\operatorname{SO}_{3}(\mathbb{R})$-orbits. 
\end{proof}

\begin{remark}
\label{rmk:S-invariant}
The restriction homomorphism $\mathrm{res}_S$ is however {\em not} surjective in this case, as opposed to the analogous situation for LMM states. To see this, note that if we write $v = (v_1, v_2, v_3)^T$, then clearly the polynomial function 
\[
f\left( \begin{tabular}{c|c}
 1 & $v^t$ \\
 \hline
 $v$ & $A$ \\ 
 \end{tabular} \right) = (v_1v_2v_3)^2 
\]
gives an invariant in $\mathbb{R}[S]^{W(S)}$ that does not come from restriction of any invariant in $\mathbb{R}[B\scr L_S]^{\operatorname{SO}_{3}(\mathbb{R})}$. It follows that $S$ is {\em not} a Chevalley section, since $\mathrm{res}_S$ is not an isomorphism. The ring of polynomial invariants $\mathbb{R}[S]^{W(S)}$ can be computed on Macaulay2 in a reasonable amount of time, but it has a quite involved presentation. This calculation, along with a quick look at the Hilbert-Molien series, suggests that the ring of invariants $\mathbb{R}[B\scr L_S]^{\operatorname{SO}_{3}(\mathbb{R})}$ is itself quite complex, and it does not seem useful for practical applications to attempt to identify its image in $\mathbb{R}[S]^{W(S)}$ directly. 
\end{remark}

We thus turn our attention instead to the problem of calculating the field of rational invariants $\mathbb{R}(B\scr L_S)^{\operatorname{SO}_{3}(\mathbb{R})}$, which seems more tractable.  We first note that $S$ is a relative section:

\begin{thm}
\label{thm:relativeSectionSymmetricStates}
The subvariety $S$ is a relative section for the action of $\operatorname{SO}_{3}(\mathbb{R})$ on $B\scr L_S$. Consequently, there is an isomorphism
\[
\mathbb{R}(B\scr L_S)^{\operatorname{SO}_{3}(\mathbb{R})} \cong \mathbb{R}(S)^{W(S)} 
\]
given by the restriction map.
\end{thm}

\begin{proof}
For ease of notation, let $G = \operatorname{SO}_{3}(\mathbb{R})$ and $X = B\scr L_S$. As in the proof of Proposition \ref{prop:injectionSymStates}, we have that $GS = X$, so certainly property (i) of a relative section holds (Definition \ref{def:relativeSection}). For the second property, let $X_0 \subseteq X$ be the dense open subset consisting of elements whose 2-point correlation matrix $A$ has distinct eigenvalues (this is the complement of the zero locus of the discriminant $\Delta$ of the characteristic polynomial of $A$, thus Zariski-open in $X$). Let $S_0 = S\cap X_0$ be the open subset  consisting of symmetric states whose 2-point correlation matrix $A$ is not only diagonal, but it has distinct entries (i.e., distinct eigenvalues). Now for any field $K\supseteq \mathbb{R}$, the only way for $gS_0(K) \cap S_0(K)$ to be non-empty is for $g\in G$ to be a signed permutation matrix, since the three entries of $A$ are all distinct. But then $g \in W(S)(K)$, and property (ii) is satisfied. Therefore $S$ is a relative section, and the isomorphism of function fields follows from the general statement given in Theorem \ref{thm:relativeSectionsAreBirational}.
\end{proof}

A useful refinement of Theorem \ref{thm:relativeSectionSymmetricStates} can be derived from a close inspection of its proof. Using the notation of the proof, $G = \operatorname{SO}_{3}(\mathbb{R})$ and $X = B\scr L_S$, let 
\[
S_0 = \left\{ \left( \begin{tabular}{c|c}
 1 & $v^t$ \\
 \hline
 $v$ & $A$ \\ 
 \end{tabular} \right) \in B\mathscr L_S: A \text{ is diagonal with distinct entries} \right \}.
\]
This is an affine open subvariety of $S$ consisting of the complement of the zero locus of $\Delta$, the discriminant of the characteristic polynomial of $A$. Because $\Delta$ is a polynomial $G$-invariant (hence a $W(S)$-invariant as well), we can localize the coordinate rings $\mathbb{R}[S]^{W(S)}$ and $\mathbb{R}[X]^G$ at $\Delta$ (applying Lemma \ref{lemma:localizationAndInvariants}) and obtain a restriction map
\begin{equation}
\label{eqn:delta-iso}
\mathrm{res}: \mathbb{R}[X]^G[\Delta^{-1}] \longrightarrow \mathbb{R}[S]^{W(S)}[\Delta^{-1}]
\end{equation}
of localized rings. We claim that this is an isomorphism:

\begin{thm}
\label{thm:delta-iso}
The map \eqref{eqn:delta-iso} is an isomorphism. 
\end{thm}

\begin{proof}
Let $Y = \Spec(\mathbb{R}[X]^G[\Delta^{-1}])$ and let $Z = \Spec(\mathbb{R}[S]^{W(S)}[\Delta^{-1}])$. The map \eqref{eqn:delta-iso} defines a morphism of affine varieties $f:Z \rightarrow Y$.  We first claim that this morphism is surjective. To show that, it suffices to show that the corresponding map of geometric points $f: Z(\mathbb{C}) \rightarrow Y(\mathbb{C})$ is surjective. To see this, note first that $Y$ is a geometric quotient of the $G$-action on $X[\Delta^{-1}]$. Indeed, over the complex numbers the orbits for this action are conjugacy classes of the 2-point correlation matrix $A$. Since $\Delta \neq 0$, all the eigenvalues of $A$ are distinct, so that two matrices are in the same conjugacy class if and only if their characteristic polynomials are equal. Moreover, the characteristic polynomial of $A$ equals the minimal polynomial, so that the conjugacy class of $A$ corresponds to the zero locus of the characteristic polynomial of $A$, and is thus a closed subset. Therefore, the orbits of the $G$-action are closed, and the algebraic quotient $Y$ is a geometric quotient \cite[I.1.3]{GIT}. The same is true for $Z$, since in this case the group $W(S)$ acting on the affine variety $S_0$ is finite and the orbits are therefore closed. A geometric point $y \in Y(\mathbb{C})$ thus corresponds to an $\mathrm{SO}_3(\mathbb{C})$-orbit consisting of points of the form 
\[
y = \left( \begin{tabular}{c|c}
 1 & $v^t$ \\
 \hline
 $v$ & $A$ \\ 
 \end{tabular} \right)
\]
such that the matrix $A$ is symmetric and has distinct eigenvalues. Because $A$ has distinct eigenvalues, we can diagonalize it over $\mathbb{C}$ using a matrix $Q$ whose columns are eigenvectors corresponding to distinct eigenvalues, thus orthogonal. It follows that $Q$ is an orthogonal matrix and moreover we can assume that it has determinant one, as in the proof of Proposition \ref{prop:injectionSymStates}. In this way we obtain a diagonal matrix $A_0 = QAQ^t$ and a geometric point

\[
z = \left( \begin{tabular}{c|c}
 1 & $(Qv)^t$ \\
 \hline
 $Qv$ & $A_0$ \\ 
 \end{tabular} \right)
\]
of $S_0$ whose $W(S)$-orbit is a geometric point of $Z(\mathbb{C})$ such that $f(z) = y$. So $f$ is surjective on geometric points, and thus surjective. 

Now let $y \in Y$, and let $\mathcal{O}_{Y,y}$ be the local ring of $Y$ at the point $y$. Let $z\in Z$ be such that $f(z) = y$ and consider the map of local rings $f_y: \mathcal{O}_{Y,y}\rightarrow \mathcal{O}_{Z,z}$ induced by $f$. Being a localization of the ring homomorphism $\mathbb{R}[X]^G \rightarrow \mathbb{R}[S]^{W(S)}$, which is injective by Proposition \ref{prop:injectionSymStates}, the homomorphism $f_y$ is also injective. The morphism $f: Z\rightarrow Y$ is moreover radicial (universally injective), as follows from property (iii) of $S$ being a relative section. By \cite{EGA1}, I.3.5.8 it follows that the extension of residue fields $k(y) \rightarrow k(z)$ induced by the map $f_y$ is radicial (purely inseparable). Because we are working in characteristic zero, it follows that $f_y$ induces an isomorphism $k(y) \cong k(z)$. By Nakayama's lemma, the homomorphism $f_y$ of local rings is thus surjective. We have now shown that for all $y\in Y$, the homomorphism $f_y: \mathcal{O}_{Y,y}\rightarrow \mathcal{O}_{Z,z}$ of local rings induced by $f$ is an isomorphism. Thus $f$ induces an isomorphism of sheaves $\mathcal{O}_Y\cong f_*\mathcal{O}_Z$ and the theorem follows by taking global sections. 
\end{proof}

\begin{example}
\label{example:restrictionCalc}
To illustrate Theorem \ref{thm:delta-iso}, consider again the $W(S)$-invariant \[
f\left( \begin{tabular}{c|c}
 1 & $v^t$ \\
 \hline
 $v$ & $A$ \\ 
 \end{tabular} \right) = (v_1v_2v_3)^2 
\] 
of Remark \ref{rmk:S-invariant}. Even though this invariant does not come from restriction of any polynomial $G$-invariant, Theorem \ref{thm:delta-iso} implies that there must be a rational $G$-invariant $g$, containing only powers of $\Delta$ in the denominators, such that $f$ is the restriction of $g$. Explicitly, let 
\[
g(v,A) := \sum_{i,j,k,\ell,m,n} \epsilon_{ijk} A_{j\ell} A_{km} A_{mn}v_i v_\ell v_n\, ,
\]
where $i,j,k$ are distinct. If $A$ is a diagonal matrix with diagonal entries $A_{ii} = \lambda_i$, then a straightforward calculation 
yields
\[
g(v,A)  = v_1v_2v_3 (\lambda_1 - \lambda_2) (\lambda_1 - \lambda_3) (\lambda_2 - \lambda_3)\, ,
\]
so that 
\[
r = \frac{g(v,A)^2}{\Delta} \in   \Bbb R[B\scr L_S]^G[\Delta^{-1}] \, .
\]
restricts to the $W(S)$-invariant $f$. The expression for $g$ was found via a graphical calculus analogous to that of \cite{King}. 
\end{example}

Note that in the proof of Theorem \ref{thm:delta-iso}, we have used the following construction of a geometric rational quotient for the action of $\operatorname{SO}_{3}(\mathbb{R})$ on symmetric states, which we record below:  

\begin{thm}
\label{thm:geometricQuotientSym}
Let $X = B\scr L_S$, $G = \operatorname{SO}_{3}(\mathbb{R})$ and let 
\[
X_0 = \left\{ \left( \begin{tabular}{c|c}
 1 & $v^t$ \\
 \hline
 v & $A$ \\ 
 \end{tabular} \right)  : A \text{ has distinct eigenvalues} \right\} \subseteq X. 
\] 
Let $\Delta$ be the discriminant of the characteristic polynomial of $A$. Then the algebraic quotient \[X_0/\!/ G = \Spec( \mathbb{R}[X_0]^G ) = \Spec( \mathbb{R}[X]^G[\Delta^{-1}])\]
is a geometric  quotient of $X_0$, and therefore a geometric rational quotient for the action of $G$ on $X$. 
\end{thm}

Finally, we can now prove that the field of rational invariants is purely transcendental:

\begin{thm}
\label{thm:rationalitySym}
The field of rational invariants $\mathbb{R}(B\scr L_S)^{\operatorname{SO}_{3}(\mathbb{R})}$ is purely transcendental of degree six. 
\end{thm}

\begin{proof}
We use the version of the No-Name Lemma given in Corollary \ref{cor:rationality}, applied to the relative section $S$ of Theorem \ref{thm:relativeSectionSymmetricStates}. 
In what follows, set $V := S$; then $\dim_{\mathbb{R}} V = 6$. Let $H := W(S) = \mathrm{O}$ be the chiral octahedral group of signed $3\times 3$ permutation matrices. 
We then have a decomposition $V = V_1\oplus V_2$, where 
\[
V_1 = \left\{ \left( \begin{tabular}{c|ccc} 
 1 & $v_1$ & $v_2$ & $v_3$ \\
 \hline
 $v_1$ & 0 & 0 & 0  \\ 
 $v_2$ & 0 & 0 & 0  \\ 
 $v_3$ & 0 & 0 & 0  \\ 
 \end{tabular} \right)\right\},  \quad  
 V_2 = \left\{ \left( \begin{tabular}{c|ccc} 
 1 & 0 & 0 & 0 \\
 \hline
0 & $a_1$ & 0 & 0  \\ 
 0 & 0 & $a_2$ & 0  \\ 
 0 & 0 & 0 & $a_3$  \\ 
 \end{tabular} \right)\right\}
\] 
are the $H$-invariant subspaces of $S$ corresponding to the subspaces of 1-point and 2-point correlation functions. The action of $H$ on $V_1$ is the standard action and it is almost free, since over geometric points it is free on the open subset $U$ given by 
\[
\prod_{i<j} (v_i - v_j) \ne 0 \, .
\]
It can be easily shown that the field of rational invariants $\mathbb{R}(V_1)^H$ is purely transcendental of degree three (see Lemma \ref{lem:rationalityOctahedral} below for a proof). It now follows from Corollary \ref{cor:rationality} that $\mathbb{R}(V)^H$ itself is purely transcendental of degree six, and therefore the same is true for $\mathbb{R}(B\scr L_S)^{\operatorname{SO}_{3}(\mathbb{R})}$, by Theorem \ref{thm:relativeSectionsAreBirational}.
\end{proof}

\begin{remark}
As in immediate Corollary, we deduce the fact that the geometric rational quotient of Theorem \ref{thm:geometricQuotientSym} is  birationally equivalent to $\mathbb{P}_{\mathbb{R}}^6$. 
\end{remark}

\subsection{Effective rationality}
\label{subsec:effective}
 By Thm. \ref{thm:rationalitySym}, the field of invariants of symmetrically mixed states of two qubits is generated by six algebraically independent rational invariants. We now explain how to make this result effective, that is, we derive explicit expressions for six generating invariants in terms of a choice of coordinates for Liouville space.

Let 
$$
G = \operatorname{O} := \{ g \in \mathrm{GL}(3,\mathbb{C}), ge_i = \pm e_j, \det(g) = 1 \}
$$
be the chiral octahedral group, introduced in Section \ref{section:LMMstates}. 
Let $V = \mathbb{C}^3$, with the standard action of $G$. The ring of polynomial invariants $\mathbb{C}[V]^G$ has a simple presentation as follows. Write $v = (v_1,v_2,v_3)^t \in V$ and consider the polynomial
$$
p(x) = \prod_{i=1}^3 (x - v_i^2) = x^3 - p_1x^2 + p_2x - p_3, 
$$ 
so that 
\begin{align*}
p_1(v) &= v_1^2 + v_2^2 + v_3^2 \\
p_2(v) &= (v_1v_2)^2 + (v_1v_3)^2 + (v_2v_3)^2 \\
p_3(v) &= (v_1v_2v_3)^2
\end{align*}
are the elementary symmetric polynomials in the squares of the coordinates of $v$. Let 
$$
p_4(v) = v_1v_2v_3\,\mathrm{disc}(p) = v_1v_2v_3(v_1^2 - v_2^2)(v_1^2 - v_3^2)(v_2^2 - v_3^2)
$$
Then $p_1, p_2, p_3, p_4$ generate $\mathbb{C}[V]^G$, as can be checked algorithmically using Macaulay2. The invariants $p_1, p_2, p_3$ form a set of primary invariants for $\mathbb{C}[V]^G$, and $p_4$ is a secondary invariants satisfying a relation of the form 
\begin{equation}
\label{eqn:relation}
p_4^2 = P_9(p_1, p_2, p_3)
\end{equation}
where $P_9$ is a quasi-homogeneous polynomial with rational coefficients, total degree 9 and weights $1,2,3$, respectively.

\begin{lem}
\label{lem:rationalityOctahedral}
The field of rational invariants $\mathbb{C}(V)^G$ is purely transcendental of degree three, generated by the three rational invariants 
$$
X = p_2/p^2_1, \quad Y = p_3/p_1^3,\quad  Z = p_4/p_1^4.
$$
\end{lem}

\begin{proof}
$\mathbb{C}(V)^G$ is a purely transcendental extension of $\mathbb{C}(\mathbb{P}V)^G \subseteq \mathbb{C}(\mathbb{P}V)$, which is a unirational field of transcendence degree 2, thus purely transcendental by Castelnuovo's Theorem. To find generators for $\mathbb{C}(V)^G$, let $X,Y,Z$ as in the statement of the Lemma. Then by \eqref{eqn:relation}, we have 
$$
\frac{p_4^2}{p_1^9} = P_9(1,X,Y),
$$
a polynomial in the rational invariants $X,Y \in \mathbb{C}(V)^G$. Note that $p_1 = Z^2/P_9(1,X,Y)$ belongs to $\mathbb{C}(X,Y,Z)$, so that every invariant $p_1, p_2, p_3, p_4$ can be expressed as a rational function in $X,Y,Z$. Therefore $\mathbb{C}[V]^G \subseteq \mathbb{C}(X,Y,Z)$, so that 
$$
\mathrm{Frac}(\mathbb{C}[V]^G) = \mathbb{C}(V)^G \subseteq \mathbb{C}(X,Y,Z)
$$
which shows $\mathbb{C}(X,Y,Z) = \mathbb{C}(V)^G$. 
\end{proof}

\begin{remark}
Since $X,Y,Z$ have $\mathbb{Q}$-coefficients, we also have 
$$
\mathbb{R}(V)^G = \mathbb{R}(X,Y,Z)
$$
\end{remark}

Let now $S$ be the relative section for symmetric states given above, where the matrix $A$ of 2-point correlation functions is diagonal. The Weyl group $W(S)$ is the chiral octahedral group $G$, acting via the standard representation $V$ on the subspace of 1-point correlation functions.  Recall that restriction gives an isomorphism 
$$
\mathbb{R}(S)^{W(S)} \simeq \mathbb{R}(B\mathcal{L}_{\mathrm{sym}})^{\mathrm{SO}_3(\mathbb{R})}
$$
at the level of rational invariants. In particular, the three invariants $X,Y,Z$ generating the field of rational invariants of the octahedral group lift to three rational invariants 
$$
p_X, p_Y, p_Z \in \mathbb{R}(B\mathscr{L}_{\mathrm{sym}})^{\mathrm{SO}_3(\mathbb{R})}.
$$
These invariants can be calculated explicitly as in Example \ref{example:restrictionCalc}. They are all rational invariants only requiring powers of the discriminant $\Delta$ of $A$ in the denominator. This set of invariants can then be extended to a set of generating invariants, as follows:

\begin{thm}
\label{thm:effectiveRationalitySymmetricStates}
The six invariants 
$$
p_X, p_Y, p_Z, \tr(A), \tr(A^2), \det(A)
$$
generate the field $\mathbb{R}(B\mathscr{L}_{\mathrm{sym}})^{\mathrm{SO}_3(\mathbb{R})}$.
\end{thm}

\begin{proof}
It suffices to show that the six invariants separate geometric orbits in general position. This follows from \cite[Lemma 2.1]{PV} for the case $k= \mathbb{C}$, but the argument can easily be extended to an arbitrary field of characteristic zero by invoking \cite{EGA1}, I.3.5.8.  Let $U = X_0$ be the open subset of $B\mathscr{L}_{\mathrm{sym}}(\mathbb{C})$ where $\Delta \neq 0$, and let $s = (v,A)$, $s' = (v', A')$ be symmetric states such that all the six invariants above are equal when evaluated on $s,s'$. We need to show that these states are in the same $\mathrm{SO}_3(\mathbb{R})$-orbit. For this we may assume that $A = \mathrm{diag}(\lambda_1, \lambda_2, \lambda_3)$ and $A' = \mathrm{diag}(\lambda'_1, \lambda'_2, \lambda'_3)$, since all eigenvalues of $A,A'$ are distinct, so that $s,s'$ belong to the section $S$. Recall that the invariants $p_X, p_Y, p_Z$, by construction, restrict to the invariants $X,Y,Z$, which separate the orbits in general position of the standard representation of the Weyl group $W(S)$. Therefore, after replacing $U$ by a possibly smaller open subset, we may assume that $v = \gamma v'$, for some $\gamma \in W(S)$. Note that the Weyl group $W(S)$ preserves the section $S$, so that $A,A'$ stay diagonal after conjugation by $\gamma$. Therefore we may assume that $v = v'$ and that $A, A'$ are diagonal. But then the entries of $A$ must be the same as the entries of $A'$, up to a permutation, since $\tr(A), \tr(A^2), \det(A)$ are the same on $A$ and $A'$. This shows that $s,s'$ are in the same orbit. 

\end{proof}

\begin{remark}
The ideas of Sections \ref{section:LMMstates} and \ref{section:symStates} can be combined to prove rationality for the full space of mixed state of two qubits. In particular, it can be shown that when $X = \mathscr{L}$ is the full Liouville space of two mixed qubits, the subspace  

\[
S = \left\{ \left( \begin{tabular}{c|c}
 1 & $v^t$ \\
 \hline
 $u$ & $C$ \\ 
 \end{tabular} \right) \in B\mathscr L : C  \text{ is diagonal  }  \right\} \subseteq B\mathscr L
\]
is a relative section, whose Weyl group is a non-abelian group of order 96. The invariant open subset where the discriminant of the characteristic polynomial of $CC^t$ is non-vanishing can be used to construct explicit geometric quotients. With some work, this approach leads to a proof of rationality and provides explicit generators for the field of rational invariants, although these expressions are extremely complicated, and of little practical use.    

In the forthcoming work \cite{SecondPaper}, we provide instead a different relative section $S'\subseteq \mathscr{L}$ with {\em abelian} Weyl group. This relative section cannot be employed to construct geometric quotients, but it leads to a simpler proof of rationality and to simpler generators for the field of invariants. Moreover, it lends itself to generalizations to a higher number of qubits. For this reason, we chose to treat this case separately, while referring to the methods of the present article. 

\end{remark}


\begin{thebibliography}{FGG{\etalchar{+}}20}

\bibitem[Bry02]{Brylinski}
Jean-Luc Brylinski.
\newblock Algebraic measures of entanglement.
\newblock In {\em Mathematics of quantum computation}, Comput. Math. Ser.,
  pages 3--23. Chapman \& Hall/CRC, Boca Raton, FL, 2002.
  
  
\bibitem[CCKR23]{SecondPaper}
Luca Candelori, Vladimir Chernyak, John Klein, Nick Rekuski
\newblock Effective Rationality for Local Unitary Invariants of Mixed States of Two Qubits
\newblock preprint, https://arxiv.org/abs/2305.16178.

\bibitem[CH87]{Choudury-Horn}
Dipa Choudhury and Roger~A. Horn.
\newblock An analog of the singular value decomposition for complex orthogonal
  equivalence.
\newblock {\em Linear and Multilinear Algebra}, 21(2):149--162, 1987.

\bibitem[Cra69]{Craven}
B.~D. Craven.
\newblock Complex symmetric matrices.
\newblock {\em J. Austral. Math. Soc.}, 10:341--354, 1969.

\bibitem[CTS07]{CT-S}
Jean-Louis Colliot-Th\'{e}l\`ene and Jean-Jacques Sansuc.
\newblock The rationality problem for fields of invariants under linear
  algebraic groups (with special regards to the {B}rauer group).
\newblock In {\em Algebraic groups and homogeneous spaces}, volume~19 of {\em
  Tata Inst. Fund. Res. Stud. Math.}, pages 113--186. Tata Inst. Fund. Res.,
  Mumbai, 2007.

\bibitem[DK15]{CIT}
Harm Derksen and Gregor Kemper.
\newblock {\em Computational invariant theory}, volume 130 of {\em
  Encyclopaedia of Mathematical Sciences}.
\newblock Springer, Heidelberg, enlarged edition, 2015.
\newblock With two appendices by Vladimir L. Popov, and an addendum by Norbert
  A'Campo and Popov, Invariant Theory and Algebraic Transformation Groups,
  VIII.

\bibitem[Dol87]{dolgachev}
Igor~V. Dolgachev.
\newblock Rationality of fields of invariants.
\newblock In {\em Algebraic geometry, {B}owdoin, 1985 ({B}runswick, {M}aine,
  1985)}, volume~46 of {\em Proc. Sympos. Pure Math.}, pages 3--16. Amer. Math.
  Soc., Providence, RI, 1987.

\bibitem[FGG{\etalchar{+}}20]{invariantRing}
Luigi Ferraro, Federico Galetto, Francesca Gandini, Hang Huang, Matthew
  Mastroeni, and Xianglong Ni.
\newblock The invariantring package for macaulay2.
\newblock \url{https://arxiv.org/abs/2010.15331}, 2020.

\bibitem[Gam16]{Gamel}
Omar Gamel.
\newblock Entangled {B}loch spheres: {B}loch matrix and two-qubit state space.
\newblock {\em Phys. Rev. A}, 93(6):062320, 18, 2016.

\bibitem[GKP16]{X-states}
V.~Gerdt, A.~Khvedelidze, and Yu. Palii.
\newblock On the ring of local unitary invariants for mixed {$X$}-states of two
  qubits.
\newblock {\em Zap. Nauchn. Sem. S.-Peterburg. Otdel. Mat. Inst. Steklov.
  (POMI)}, 448(Teoriya Predstavleni\u{\i}, Dinamicheskie Sistemy, Kombinatornye
  Metody. XXVII):107--123, 2016.

\bibitem[Gro60]{EGA1}
A.~Grothendieck.
\newblock \'{E}l\'{e}ments de g\'{e}om\'{e}trie alg\'{e}brique. {I}. {L}e
  langage des sch\'{e}mas.
\newblock {\em Inst. Hautes \'{E}tudes Sci. Publ. Math.}, (4):228, 1960.

\bibitem[GS]{M2}
Daniel~R. Grayson and Michael~E. Stillman.
\newblock Macaulay2, a software system for research in algebraic geometry.
\newblock Available at \url{http://www.math.uiuc.edu/Macaulay2/}.

\bibitem[HHHH09]{Horodecki}
Ryszard Horodecki, Pawel Horodecki, Michal Horodecki, and Karol Horodecki.
\newblock Quantum entanglement.
\newblock {\em Rev. Modern Phys.}, 81(2):865--942, 2009.

\bibitem[JKW07]{King}
P.D. Jarvis, R.C. King, and T.A. Welsh.
\newblock The mixed two-qubit system and the structure of its ring of local
  invariants.
\newblock {\em J. Phys. A}, 40(33):10083--10108, 2007.

\bibitem[MFK94]{GIT}
D.~Mumford, J.~Fogarty, and F.~Kirwan.
\newblock {\em Geometric invariant theory}, volume~34 of {\em Ergebnisse der
  Mathematik und ihrer Grenzgebiete (2) [Results in Mathematics and Related
  Areas (2)]}.
\newblock Springer-Verlag, Berlin, third edition, 1994.

\bibitem[MW02]{Wallach-Meyer}
David~A. Meyer and Noland Wallach.
\newblock Invariants for multiple qubits: the case of 3 qubits.
\newblock In {\em Mathematics of quantum computation}, Comput. Math. Ser.,
  pages 77--97. Chapman \& Hall/CRC, Boca Raton, FL, 2002.

\bibitem[NC00]{NielsenChuang}
Michael~A. Nielsen and Isaac~L. Chuang.
\newblock {\em Quantum computation and quantum information}.
\newblock Cambridge University Press, Cambridge, 2000.

\bibitem[Ros56]{Rosenlicht}
Maxwell Rosenlicht.
\newblock Some basic theorems on algebraic groups.
\newblock {\em Amer. J. Math.}, 78:401--443, 1956.

\bibitem[Sha94]{PV}
I.~R. Shafarevich, editor.
\newblock {\em Algebraic geometry. {IV}}, volume~55 of {\em Encyclopaedia of
  Mathematical Sciences}.
\newblock Springer-Verlag, Berlin, 1994.
\newblock Linear algebraic groups. Invariant theory, A translation of {{\i}t
  Algebraic geometry. 4} (Russian), Akad. Nauk SSSR Vsesoyuz. Inst. Nauchn. i
  Tekhn. Inform., Moscow, 1989 [ MR1100483 (91k:14001)], Translation edited by
  A. N. Parshin and I. R. Shafarevich.

\bibitem[Wal17]{Wallach-book}
Nolan~R. Wallach.
\newblock {\em Geometric invariant theory}.
\newblock Universitext. Springer, Cham, 2017.
\newblock Over the real and complex numbers.

\end{thebibliography}

\newcommand{\etalchar}[1]{$^{#1}$}

\end{document}